\documentclass[a4paper,11pt,reqno]{amsart}

\usepackage{xcolor}

\usepackage{amsmath,amsthm,array,xypic}
\usepackage{amssymb}
\usepackage{amsrefs}
\usepackage{mathtools}
\usepackage{enumerate}
\usepackage[mathcal]{euscript}
\usepackage[margin=3cm]{geometry}

\usepackage[utf8]{inputenc}
\usepackage{tgtermes}
\usepackage{mathptmx}
\usepackage[T1]{fontenc}
\usepackage{stackrel}
\usepackage{pict2e,picture}
\usepackage{hyperref}

\newtheorem{theorem}{Theorem}[section]
\newtheorem{lemma}[theorem]{Lemma}
\newtheorem{proposition}[theorem]{Proposition}

\theoremstyle{definition}

\newtheorem{remark}[theorem]{Remark}

\newcommand{\CC}{\mathbb{C}}

\newcommand{\HH}{\operatorname{HH}}
\newcommand{\Hom}{\operatorname{Hom}}

\newcommand{\Mat}{\operatorname{Mat}}

\newcommand{\op}{{\operatorname{o}}}

\newcommand{\Ext}{\operatorname{Ext}}

\newcommand{\id}{\operatorname{id}}

\renewcommand{\d}{\operatorname{{\partial}}}

\newcommand{\leftbimod}[3]{\vphantom{#1}^{#2}{\kern-#3pt #1}}

\newcommand{\<}{\<}
\renewcommand{\>}{\>}

\newcommand{\lA}{\widetilde{A}}
\newcommand{\lf}{\widetilde{f}}

\numberwithin{equation}{subsection}

\title{Hochschild cohomology and lifts of endomorphisms}

\author{Niels Lauritzen and Jesper Funch Thomsen}
\address{Department of Mathematics,
Aarhus University,
Ny Munkegade 118,
DK-8000 Aarhus C,
Denmark }
\email{niels@math.au.dk, funch@math.au.dk}

\begin{document}
\maketitle

\begin{abstract}

We study when algebra endomorphisms can be lifted to first--order flat lifts.  
To a first--order flat lift of an algebra and an endomorphism, we associate a 
canonical class in Hochschild cohomology with coefficients in a
naturally twisted bimodule. The cohomology class 
vanishes
exactly when the endomorphism admits a multiplicative lift.

For an Azumaya algebra of constant rank over a formally smooth center, we prove that 
an endomorphism lifts if and only if the 
induced endomorphism of the center preserves the Poisson structure given 
by the lift of the algebra.
\end{abstract}

\section{Introduction}

A fundamental theme in deformation quantization is the interplay between
a noncommutative algebra and the Poisson geometry of its semiclassical limit.
In \cite{BK}, Belov-Kanel and Kontsevich
conjectured that the automorphism group of the $n$-th Weyl algebra $A_n(\CC)$ 
is canonically isomorphic to the group of polynomial symplectomorphisms 
of $\CC^{2n}$, equipped with its standard Poisson structure.
Their key technique proceeds by 
reducing the Weyl algebra to positive characteristic $p$, 
where $A_n(k)$ becomes an Azumaya algebra whose center 
$Z \cong k[x_1^p, \ldots, x_n^p, \d_1^p, \ldots, \d_n^p]$ 
is a polynomial algebra. Here $A_n(k)$ admits 
a canonical (flat) lift to $A_n(W_2(k))$ --- the 
Weyl algebra over the Witt vectors $W_2(k)$  of length two. The existence of a lift of 
a Weyl algebra endomorphism $f$ of $A_n(k)$ to $A_n(W_2(k))$ shows that the
$f$ is a symplectomorphism when restricted to the center.
Belov-Kanel and Kontsevich exploit this structure to construct a homomorphism from the automorphism 
group of the Weyl algebra to the group of Poisson automorphisms of the center.
The same circle of ideas, developed independently by Tsuchimoto~\cite{Ts}, 
also establishes the stable equivalence of the Jacobian and Dixmier conjectures~\cite{BKJ}.

In \cite{LTL} we work entirely over a perfect positive characteristic field $k$ and show that 
a $k$ -algebra endorphism of $A_n(k)$ admits a lift to $A_n(W_2(k))$ 
if and only if it preserves the canonical Poisson bracket on the center.
 The goal of this paper is to generalize that result in the framework of 
 \emph{Hochschild cohomology}.

We work in the following general setup. 
For any field $k$ and any $k$-algebra $A$, an endomorphism $f$ of $A$ together with a first-order flat lift of $A$ determines a canonical class in Hochschild cohomology of the bimodule $A$
twisted by $f$ from the left and right; after choosing an $R$-linear lift of $f$, this class is represented by a natural $2$-cocycle.
This cohomology class vanishes if and only if 
the endomorphism lifts to an endomorphism of the lift of $A$. 
Our main observation is that if $f$ preserves the center and the 
Poisson bracket on the center given by the lift of $A$, then the restriction to the center of the associated $2$-cocycle is symmetric. 

If $A$ is an Azumaya algebra of constant rank over a formally smooth center $Z$ in this setup, we prove in Theorem \ref{theorem:main} that $f$ admits 
a lift if and only if $f$ preserves the Poisson bracket on $Z$. The condition $f(Z)\subseteq Z$ is automatic
in this case --- see Proposition \ref{proposition:Azumaya} for the general statement.

We thank M.~B\"okstedt for patiently answering our questions about Hochschild cohomology.

\section{Notation and preliminaries}

In the following, $k$ will denote a field and $A$ a $k$-algebra. 
We denote by $A^\op$ the opposite algebra of $A$
and put $A^e = A\otimes_k A^\op$. The bar complex of $A$ is the complex $(B_n(A), d_n)$ of 
free $A^e$-modules given by
$B_n(A) = A^{\otimes_k (n+2)}$ and
$$
d_n(a_0 \otimes a_1 \otimes \cdots \otimes a_n\otimes a_{n+1}) =
\sum_{i=0}^n (-1)^i a_0\otimes a_1 \otimes \cdots \otimes a_i a_{i+1} \otimes \cdots \otimes a_{n+1}
$$
for $n\geq 0$. It is a free resolution of $A$ by the multiplication map $B_0(A) \rightarrow A$. For an
$A$-bimodule $M$, there is a canonical $k$-isomorphism 
\begin{equation}\label{HHExt}
\Hom_{A^e}(B_n(A), M) \rightarrow  \Hom_k(A^{\otimes_k n}, M) =: C^n(A/k, M)
\end{equation}
given by $h\mapsto (a_1\otimes \cdots \otimes a_n \mapsto h(1\otimes a_1 \otimes \cdots \otimes a_n\otimes 1))$ for $n\geq 0$.
The complex $C^\bullet(A/k, M)$ is called 
the \emph{Hochschild complex} of $M$. It has differentials given by
\begin{align*}
(\delta^n f)(a_1 \otimes \cdots \otimes a_{n+1}) &= a_1 \cdot f(a_2 \otimes \cdots \otimes a_{n+1}) \\
&\quad + \sum_{i=1}^{n} (-1)^i f(a_1 \otimes \cdots \otimes a_i a_{i+1} \otimes \cdots \otimes a_{n+1}) \\
&\quad + (-1)^{n+1} f(a_1 \otimes \cdots \otimes a_n) \cdot a_{n+1}
\end{align*}
for $n\geq 0$.
The $n$-th Hochschild cohomology of $A$ with coefficients in $M$ is defined as
$$
\HH^n(A/k,M) := H^n(C^\bullet(A/k,M))
$$
for $n \geq 0$. The isomorphisms \eqref{HHExt} is a morphism of complexes and we have the following result.

\begin{proposition}
$$
\HH^i(A/k, M) = \Ext^i_{A^e}(A, M).
$$
\end{proposition}

We call an $A$-bimodule $M$, 
$B$-diagonal if $b m = m b$ for every $b\in B$, where $B$ is a subalgebra of $A$.

\section{First order flat lifts of endomorphisms}

Let $R$ be a commutative local ring with non-zero principal maximal ideal $pR$ such that $p^2 = 0$, and let $k = R/pR$ be its residue field. 
The map $k \to pR$ given by $r +p R \mapsto pr$ is a well-defined isomorphism; we denote its inverse by $p^{-1}: pR \to k$.

Let $A$ be a $k$-algebra. By a first-order flat lift of $A$ we mean a flat $R$-algebra $\lA$ together with an identification
$\lA\otimes_R k \cong A$. We let $\pi:\lA\rightarrow A$ denote the reduction map, and for $a\in A$ we write $\tilde a\in \lA$ for an arbitrary lift of $a$.
Since $p^2=0$, the $\lA$-bimodule $p\lA=\ker(\pi)$ is annihilated by $p$, so its left and right $\lA$-actions factor through $A=\lA/p\lA$; hence $p\lA$ is naturally an $A$-bimodule.
Flatness ensures that multiplication by $p$ 
induces an $A$-bimodule isomorphism between $A$ and $p\lA$. Explicitly, its inverse is given by:
$$
p^{-1}: p\lA \longrightarrow A, \quad p \tilde{a} \longmapsto \pi(\tilde{a}) = a,
$$
and we denote by $\iota: A \to p\lA$ the inverse map $a \mapsto p\tilde{a}$.

\begin{remark}
Since $R$ is artinian local, every flat $R$-module is free. In particular, $\lA$ is a free $R$-module. 
In fact, if $\{e_i\}_{i \in I}$ is a $k$-basis for $A$, then $\{\tilde{e}_i\}_{i \in I}$ forms an $R$-basis for $\lA$.
\end{remark}

Now, consider a $k$-algebra homomorphism $f: A \to A$ and an $R$-linear first order lift $\lf: \lA \to \lA$ (satisfying $\pi \circ \lf = f \circ \pi$). 
We define the multiplicative defect $D: \lA \times \lA \to \lA$ by:
$$
D(x, y) = \lf(x) \lf(y) - \lf(xy).
$$
Because $p^2 = 0$ and $\lf$ is $R$-linear, the defect is invariant under adding elements of $p\lA$. 
Consequently, $D$ factors through the reduction map, yielding a well-defined map $L_{\lf}: A \otimes_k A \to p\lA$ given by:
$$
L_{\lf}(a \otimes b) := D(\tilde{a}, \tilde{b}).
$$
Thus, $\lf$ is multiplicative (and hence an $R$-algebra homomorphism) if and only if $L_{\lf}(a \otimes b) = 0$ for all $a, b \in A$.
If $\lf$ is multiplicative, then automatically $\lf(1) = 1$.

\begin{proposition}\label{proposition:cocycle}
Let $f: A \to A$ be a $k$-algebra homomorphism and $\lf: \lA \to \lA$ an $R$-linear lift of $f$. The $k$-linear map $C_{\lf} \in \Hom_k(A \otimes_k A, A)$ defined by 
$$
C_{\lf}(a \otimes b) := p^{-1} L_{\lf}(a \otimes b) 
$$ 
is a $2$-cocycle in $C^2(A/k, M)$, where $M = A$ viewed as an $A$-bimodule via the twisted action
$$
a \cdot m \cdot b := f(a) m f(b)
$$
for $m \in M$ and $a, b \in A$. 

The cohomology class $[C_{\lf}] \in H^2(A/k, M)$ is independent of the choice of the $R$-linear lift $\lf$, and $[C_{\lf}] = 0$ if and only if $f$ admits a lift to an $R$-algebra homomorphism $\widehat{f}: \lA \to \lA$.
\end{proposition}

\begin{proof}
By the definition of the Hochschild differential for the twisted bimodule $M$, we have for all $a,b,c \in A$:
\begin{align*} 
(\delta^2 C_{\lf})(a \otimes b \otimes c) &= f(a) C_{\lf}(b \otimes c) - C_{\lf}(ab \otimes c) + C_{\lf}(a \otimes bc) - C_{\lf}(a \otimes b) f(c) \\
&= p^{-1} \Big( \lf(\tilde{a}) L_{\lf}(b \otimes c) - L_{\lf}(ab \otimes c) + L_{\lf}(a \otimes bc) - L_{\lf}(a \otimes b) \lf(\tilde{c}) \Big).
\end{align*}
Expanding $L_{\lf}$ via the multiplicative defect $D$, the term inside the parenthesis becomes:
\begin{align*}
&\quad\, \lf(\tilde{a}) \big( \lf(\tilde{b})\lf(\tilde{c}) - \lf(\tilde{b}\tilde{c}) \big) - \big( \lf(\tilde{a}\tilde{b})\lf(\tilde{c}) - \lf(\tilde{a}\tilde{b}\tilde{c}) \big) \\
&+ \big( \lf(\tilde{a})\lf(\tilde{b}\tilde{c}) - \lf(\tilde{a}\tilde{b}\tilde{c}) \big) - \big( \lf(\tilde{a})\lf(\tilde{b}) - \lf(\tilde{a}\tilde{b}) \big) \lf(\tilde{c}).
\end{align*}
Due to the associativity of the algebra $\lA$, all terms cancel pairwise, yielding $(\delta^2 C_{\lf})(a \otimes b \otimes c) = 0$. Thus, $C_{\lf}$ is a $2$-cocycle.

Now, suppose $\lf_1$ and $\lf_2$ are two $R$-linear lifts of $f$. Their difference $\lf_2 - \lf_1$ maps into $p\lA$ and therefore defines a $k$-linear map $h: A \to A$ via:
$$
(\lf_2 - \lf_1)(\tilde{a}) = \iota(h(a)).
$$
Direct substitution shows that
$$
C_{\lf_2}(a \otimes b) - C_{\lf_1}(a \otimes b) = f(a)h(b) - h(ab) + h(a)f(b) = (\delta^1 h)(a \otimes b).
$$
Hence, $C_{\lf_2} - C_{\lf_1}$ is a coboundary, meaning the cohomology class $[C_{\lf}]$ is independent of the lift.

Finally, if $f$ lifts to an $R$-algebra homomorphism $\widehat{f}: \lA \to \lA$, then $L_{\widehat{f}} = 0$, making $C_{\widehat{f}} = 0$ and thus $[C_{\lf}] = 0$. Conversely, suppose $[C_{\lf}] = 0$. Then $C_{\lf} = \delta^1 h$ for some $1$-cochain $h \in \Hom_k(A,A)$, meaning:
$$
C_{\lf}(a \otimes b) = f(a)h(b) - h(ab) + h(a)f(b) \quad \forall a,b \in A.
$$
We define a new $R$-linear lift $\widehat{f}: \lA \to \lA$ by correcting $\lf$ with $h$:
$$
\widehat{f}(x) := \lf(x) - \iota(h(\pi(x))).
$$
One readily verifies that $L_{\widehat{f}}(a \otimes b) = 0$ for all $a,b \in A$, proving that $\widehat{f}$ is the desired $R$-algebra homomorphism.
\end{proof}

It is well known that a first order flat lift of $A$ gives rise to a Poisson bracket on the 
center of $A$. This is very helpful in studying the cocycle defined in Proposition \ref{proposition:cocycle}.

\subsection{Restriction of the cocycle to the center}

Let $Z$ be the center of $A$. For $z_1, z_2\in Z$, 
$$
\{z_1, z_2\} = p^{-1}([\tilde{z}_1, \tilde{z}_2]) = p^{-1}(\tilde{z}_1 \tilde{z}_2 - \tilde{z}_2 \tilde{z}_1)
$$ 
defines a Poisson bracket on the center $Z$ 
of $A$.

\begin{proposition}\label{proposition:symmetric}
    Assume that $f(Z) \subseteq Z$ and let $\lf: \lA \rightarrow \lA$ be an $R$-linear lift of 
    $f$. Let $c \in C^2(Z/k,M)$ be the restriction of $C_{\lf}$ to $Z^{\otimes_k 2}$. Then $c$ is a $2$-cocycle and
    $$
    c(x \otimes y) - c(y \otimes x) =  \{f(x), f(y)\} - f(\{x, y\}) 
    $$
    for $x, y\in Z$.
    
    In particular, 
the induced endomorphism $f: Z\rightarrow Z$ preserves
    the Poisson bracket if and only if the $2$-cocycle $c\in C^2(Z, M)$
    is symmetric, i.e., 
    $c(x \otimes y) = c(y \otimes x)$ for 
    every $x, y\in Z$.
\end{proposition}
\begin{proof}
Since the Hochschild differential commutes with restriction along $Z\hookrightarrow A$, the restriction $c$ of $C_{\lf}$ is again a $2$-cocycle.

For $x, y\in Z$,
\begin{align*}
&c(x \otimes y) - c(y \otimes x) = \\
&p^{-1}(\lf(\tilde{x}) \lf(\tilde{y}) - 
\lf(\tilde{x}\tilde{y}) - 
\lf(\tilde{y}) \lf(\tilde{x}) +
\lf(\tilde{y}\tilde{x})) = \\
&p^{-1}(\lf(\tilde{x}) \lf(\tilde{y}) - 
\lf(\tilde{y}) \lf(\tilde{x}) -
\lf(\tilde{x}\tilde{y} - \tilde{y}\tilde{x})) =\\
&\{f(x), f(y)\} - f(\{x, y\})
\end{align*}
This shows in particular that the antisymmetrization $c(x \otimes y) - c(y \otimes x)$ depends only on $f$ and the lift $\lA$, and not on the chosen $R$-linear lift $\lf$.
\end{proof}

Under favorable conditions, symmetric cocycles have trivial cohomology class in the Hochschild cohomology.

\begin{lemma}\label{lemma:formallysmooth}
Let $Z/k$ be a commutative algebra which is formally smooth over $k$ in the category of commutative $k$-algebras, and let $M$ be a diagonal $Z$-bimodule. Then every symmetric
$2$-cocycle in the Hochschild complex $C^\bullet(Z/k, M)$ is a
coboundary.
\end{lemma}

\begin{proof}
We interpret $\HH^2(Z/k, M)$ in terms of square-zero extensions. Let $\phi \in \operatorname{Hom}_k(Z^{\otimes 2}, M)$ be a $2$-cocycle. 
This defines a $k$-algebra structure on the vector space direct sum $E = Z \oplus M$, with multiplication given by
$$(a, m) \cdot (b, n) = (ab, a n + m b + \phi(a \otimes b)).$$
Because $\phi$ is a Hochschild $2$-cocycle, $E$ is an associative algebra, and the projection $E \to Z$ 
is a surjective $k$-algebra homomorphism whose kernel $M$ satisfies $M^2 = 0$.

Since $\phi$ is symmetric (i.e., $\phi(a \otimes b) = \phi(b \otimes a)$) and $M$ is a diagonal $Z$-bimodule (so $am = ma$ for all $a\in Z$ and $m\in M$), 
the multiplication on $E$ is commutative. Thus, $E$ is a commutative square-zero extension of $Z$ by $M$.

By definition, the $k$-algebra $Z$ is formally smooth if and only if every extension of $Z$ by a nilpotent ideal in the category of commutative $k$-algebras splits. Therefore, there exists a $k$-algebra section $s \colon Z \to E$. This section must be of the form $s(a) = (a, \eta(a))$ for some $k$-linear map $\eta \colon Z \to M$. Preserving multiplication requires
$$(ab, \eta(ab)) = s(ab) = s(a)s(b) = (a, \eta(a))(b, \eta(b)) = (ab, a\eta(b) + \eta(a)b + \phi(a \otimes b)),$$
which rearranges to
$$\phi(a \otimes b) = \eta(ab) - a\eta(b) - \eta(a)b.$$
This is precisely the statement that $\phi = -d\eta$ in the Hochschild complex, meaning $\phi$ is a coboundary.
\end{proof}

\section{Lifts of endomorphisms of Azumaya algebras}

If $A$ is a $k$-algebra and at the same time an Azumaya algebra over its center $Z$, then the multiplication map 
$\mu: A\otimes_Z A^\op \rightarrow A$ admits an $A^e$-module splitting $s$
and $A$ is a finitely generated projective $Z$-module (see \cite[III, Th\'eor\`eme 5.1, Th\'eor\`eme 1.4]{KnusOjanguren}). 

\begin{lemma}\label{lemma:restriction}
If $A$ is an Azumaya algebra with center $Z$ and
$M$ is an $A$-bimodule, then
the restriction map $C^i(A/k, M) \rightarrow C^i(Z/k, M)$ induces an 
injective map
$$
\HH^i(A/k,M) \hookrightarrow 
\HH^i(Z/k,M)
$$
on Hochschild cohomology.
\end{lemma}
\begin{proof}
Let $\Lambda = A\otimes_Z A^\op$. The complex $P_\bullet := A^e\otimes_{Z^e} B_{\bullet}(Z/k)$ is a projective resolution of $A^e\otimes_{Z^e} Z \cong \Lambda$. The multiplication map $\mu: \Lambda \rightarrow A$ lifts to an $A^e$-chain map $r_\bullet: P_\bullet \rightarrow B_\bullet(A/k)$ given by
$$
r_n((a\otimes b^\op)\otimes z_0 \otimes \cdots \otimes z_{n+1}) = a z_0\otimes \cdots \otimes z_{n+1}b.
$$
By adjunction, we have an isomorphism of complexes
$$
\Hom_{A^e}(P_\bullet, M) \cong \Hom_{Z^e}(B_\bullet(Z/k), M) = C^\bullet(Z, M),
$$
which upon taking cohomology yields a canonical isomorphism $\Ext^i_{A^e}(\Lambda, M) \cong \HH^i(Z/k, M)$. 

Under this identification, the map induced by $r_\bullet^*$ on cohomology corresponds precisely to the restriction map 
$\HH^i(A/k, M) \rightarrow \HH^i(Z/k, M)$ and $r_\bullet^*$ computes the pullback $\mu^*: \Ext^i_{A^e}(A, M) \rightarrow \Ext^i_{A^e}(\Lambda, M)$
for $i\geq 0$.

Since $A$ is an Azumaya algebra, $\mu$ admits an $A^e$-module splitting $s$ such that $\mu s = \id_A$. Therefore $s^* \mu^* = \id$, making $\mu^*$—and consequently the restriction map—injective.
\end{proof}

\begin{remark}
Suppose $M$ is a $Z$-diagonal and $s(1) = \sum_i x_i \otimes y_i\in A \otimes_Z A^\op$. Then the map $e_M: M \rightarrow M$ given by $e_M(m) = \sum_i x_i m y_i$ 
is a well-defined $Z$-module projection onto $M^A = \{m\in M \mid a m = m a, \text{for every }a\in A\}$. Then one may prove that the restriction map, 
followed by the homomorphism induced by $e_M$, yields an 
isomorphism
$$
\HH^i(A/k, M) \xrightarrow{\sim} \HH^i(Z/k, M^A)
$$
on Hochschild cohomology. Let $\sigma$ be a lift of $s(1) = \sum_i x_i \otimes y_i\in A \otimes_Z A^\op$ in $A^e$, and let $e = s \mu$. Then $e$ is right multiplication by $\sigma$, and the induced chain map on $P_\bullet$ is
$$
e_n(u\otimes \xi) = u \sigma \otimes \xi.
$$
Under the adjunction map, precomposition with $e_n$ corresponds to postcomposition with $e_M$ on Hochschild cochains $C^n(Z, M)$. 
Since $e_M$ is a split idempotent with image $M^A$, the image of the induced 
endomorphism on $C^\bullet(Z, M)$ is the direct summand subcomplex $C^\bullet(Z, M^A)$. Passing to 
cohomology yields the claimed isomorphism.
\end{remark}

We now prove that a unital \emph{ring} endomorphism $f$ of an Azumaya algebra of constant rank over an \emph{arbitrary} commutative ring $Z$ satisfies 
$f(Z)\subseteq Z$. We have only been able to find this result in the literature for reduced rings and therefore give a complete proof below. 
First we need the following lemma.

\begin{lemma}\label{lemma:CSA}
Let $B$ be an Azumaya algebra of degree $n$ over a commutative ring $C$, let $k$ be a field, and let
$$
u:B\to S
$$
be a ring homomorphism into a central simple $k$-algebra $S$ of degree $n$. Then the $k$-subalgebra of $S$ generated by $u(B)$ is all of $S$.
\end{lemma}

\begin{proof}
Let $T$ be the $k$-subalgebra of $S$ generated by $u(B)$. Suppose $T\neq S$. Choose a simple quotient
$$
T\twoheadrightarrow S_0,
$$
and put $K=Z(S_0)$. The image of $C$ in $S_0$ is central, hence lies in $K$. Therefore the composite
$
B\to S_0
$
induces a $K$-algebra map
$$
\theta:B\otimes_C K\to S_0.
$$
Its image is the $K$-subalgebra generated by the image of $B$, hence all of $S_0$. 

Now $B\otimes_C K$ is central simple of degree $n$ over $K$, so a nonzero surjective homomorphism from it is an isomorphism. Thus
$$
\dim_k S_0=\dim_k(B\otimes_C K)=n^2[K:k]\ge n^2
$$
contradicting that $
\dim_k S_0<\dim_k S=n^2.
$
Hence $T=S$.
\end{proof}

\begin{proposition}\label{proposition:Azumaya}
Let $A$ be an Azumaya algebra of constant degree over a commutative ring $R$, and let
$$
f:A\to A
$$
be a ring endomorphism. Then $f(R)\subseteq R$.
\end{proposition}

\begin{proof}
Let
$$
N:=\sum_{a\in A} R\,f(a)\subseteq A
$$
be the $R$-submodule generated by $f(A)$. It suffices to prove that
$$
N=A.
$$
Fix a maximal ideal $\mathfrak m\subset R$, and set
$k(\mathfrak m)=R/\mathfrak m$ and  $S=A/\mathfrak m A$.
Since $A$ has constant degree $n$, $S$ is a central simple $k(\mathfrak m)$-algebra of degree $n$.
Consider the composite
$$
\psi_{\mathfrak m}:A\xrightarrow{f}A\to S.
$$
Let $H=\psi_{\mathfrak m}(A)\subseteq S$. The image of $N$ in $S$ is exactly the $k(\mathfrak m)$-span $k(\mathfrak m)H$. Since $H$ is a subring, 
$k(\mathfrak m)H$ is precisely the $k(\mathfrak m)$-subalgebra generated by $H$. By Lemma \ref{lemma:CSA},
$$
k(\mathfrak m)H=S.
$$
Hence the image of $N$ in $A/\mathfrak m A$ is all of $A/\mathfrak m A$, i.e.
$
N+\mathfrak m A=A.
$
Since $\mathfrak m$ was arbitrary, we have $(A/N)\otimes_R k(\mathfrak m)=0$ for every maximal ideal $\mathfrak m\subset R$.
As $A/N$ is a finitely generated $R$-module, it follows that $A/N=0$, and hence $N=A$.
\end{proof}

\begin{remark}
The constant-degree hypothesis is necessary. Indeed, let $A = \CC\times \Mat_2(\CC)$. Then the ring endomorphism 
given by 
$$
(a, M) \mapsto \left(a, \begin{pmatrix} a & 0 \\ 0 & \overline{a} \end{pmatrix}\right)
$$
maps the central element $(i, 0)$ to a non-central element.
\end{remark}

Finally we state and prove the main application and result of this paper.

\begin{theorem}\label{theorem:main}
Let $k$ be a field, $A$ a $k$-algebra, $f: A \rightarrow A$ a $k$-algebra endomorphism of $A$ and $\lA$ a first order flat lift of $A$. 
Assume that $A$ is an Azumaya algebra of constant rank over its center $Z$, and that $Z$ is formally smooth over $k$ in the category of commutative $k$-algebras. Then $f$ admits a lift to an endomorphism $\lf: \lA\rightarrow \lA$ if and only
if $f$ preserves the Poisson bracket on $Z$ induced by $\lA$.
\end{theorem}
\begin{proof}
By Proposition \ref{proposition:cocycle}, the endomorphism $f$ together with any $R$-linear lift 
$\lf$ of $f$ determines a Hochschild $2$-cocycle $C_{\lf} \in C^2(A/k, M)$, 
whose cohomology class $[C_{\lf}] \in \HH^2(A/k, M)$ vanishes if and only if 
$f$ admits a multiplicative lift. Recall here that $M = A$ with the 
bimodule structure $a m b = f(a) m f(b)$ for $a, b, m\in A$.

Suppose first that $f$ admits a multiplicative lift $\widehat{f}$. Then for $x,y\in Z$ and lifts $\tilde{x},\tilde{y}\in \lA$ we have
$$
f(\{x,y\}) = p^{-1}\bigl(\widehat{f}([\tilde{x},\tilde{y}])\bigr)
= p^{-1}\bigl([\widehat{f}(\tilde{x}),\widehat{f}(\tilde{y})]\bigr)
= \{f(x),f(y)\},
$$
so $f$ preserves the Poisson bracket on $Z$.

Suppose conversely that $f$ preserves the Poisson bracket on $Z$. Since $A$ is Azumaya of constant rank, Proposition \ref{proposition:Azumaya} gives $f(Z)\subseteq Z$.
Hence $M$, restricted to $Z$, is a diagonal $Z$-bimodule. Proposition~\ref{proposition:symmetric} then shows that the restriction of $C_{\lf}$ to $Z$ is a symmetric $2$-cocycle in $C^2(Z/k, M)$.
Because $Z$ is formally smooth, Lemma~\ref{lemma:formallysmooth} implies that this restricted cocycle is a coboundary.
Finally, the restriction map $\HH^2(A/k, M) \hookrightarrow \HH^2(Z/k, M)$ is injective by Lemma \ref{lemma:restriction}, so $[C_{\lf}] = 0$ in $\HH^2(A/k, M)$, and therefore $f$ lifts.
\end{proof}

\begin{remark}
Let $k$ be a perfect field of positive characteristic and let $A = A_n(k)$ be the $n$-th Weyl algebra over $k$. Then it follows from 
Theorem \ref{theorem:main}, that an endomorphism of $A$ lifts to $A_n(W_2(k))$ if and only if $f$ preserves the canonical Poisson bracket 
on the center of $A$. This is exactly \cite[Theorem 2.7]{LTL}.
\medskip

It would be interesting to formulate criteria for the existence of lifts from $W_m(k)$ to $W_{m+1}(k)$ for $m \geq 2$. However, the framework developed in this paper does not seem to extend directly to that setting. Already for $m=2$, \(A_1(W_2(k))\) is no longer an Azumaya algebra, so the present methods do not apply in any immediate way. Nevertheless, we expect that the techniques of \cite{LTL} can be adapted to treat this problem for the Weyl algebra.

For \(p=2\), the endomorphism
\[
x \mapsto x, \qquad \partial \mapsto \partial + x^3 \partial^4
\]
of \(A_1(k)\) lifts to \(W_2(k)\), but not to \(W_3(k)\). We expect that the general lifting problem from \(k = W_1(k)\) to \(W_m(k)\) is governed by subtler obstructions, and it seems plausible that the de Rham--Witt complex plays a role in describing them.

\end{remark}


\begin{bibdiv}
\begin{biblist}


\bib{BK}{article}{
    author={Belov-Kanel, A.},
    author={Kontsevich, M.},
    title={Automorphisms of the Weyl algebra},
    journal={Lett. Math. Phys.},
    volume={74},
    date={2005},
    pages={181--199},
}

\bib{BKJ}{article}{
    author={Belov-Kanel, A.},
    author={Kontsevich, M.},
    title={The Jacobian conjecture is stably equivalent to the Dixmier conjecture},
    journal={Mosc. Math. J.},
    volume={7},
    number={2},
    date={2007},
    pages={209--218},
}



\bib{KnusOjanguren}{book}{
    author={Knus, M.-A.},
    author={Ojanguren, M.},
    title={Th\'eorie de la descente et alg\`ebres d'Azumaya},
    series={Lecture Notes in Mathematics},
    volume={389},
    publisher={Springer-Verlag},
    address={Berlin-New York},
    date={1974},
    pages={viii+160},
}

\bib{LTL}{misc}{
    author={Lauritzen, N.},
    author={Thomsen, J. F.},
    title={Lifts of endomorphisms of Weyl algebras modulo $p^2$},
    note={\url{https://arxiv.org/abs/2601.23110}},
}


\bib{Ts}{article}{
    author={Tsuchimoto, Y.},
    title={Endomorphisms of Weyl algebra and $p$-curvatures},
    journal={Osaka J. Math.},
    volume={42},
    number={2},
    date={2005},
    pages={435--452},
}



\end{biblist}
\end{bibdiv}
\end{document}